\documentclass[11pt,a4paper]{article}
\usepackage[utf8]{inputenc}
\usepackage{amsmath, amssymb, amsthm}
\usepackage{graphicx,hyperref,geometry}
\usepackage{a4wide}
\usepackage{xcolor,verbatim}

\newtheorem{theorem}[subsection]{Theorem}
\newtheorem{corollary}[subsection]{Corollary}
\newtheorem{lemma}[subsection]{Lemma}
\newtheorem{proposition}[subsection]{Proposition}

\newtheorem{fact}[subsection]{Fact}

\newtheorem{conjecture}[subsection]{Conjecture}

\newtheorem{question}[subsection]{Question}

\newcommand{\s}{$\mathrm{\check{s}}$}

\title{Quest for graphs of Frank number $3$}

\author{János Barát\thanks{Research supported by ERC Advanced Grant ”GeoScape”.} \\
\small Alfr\'ed R\'enyi Institute of Mathematics\\
\small University of Pannonia, Department of Mathematics\\
\small 8200 Veszprém, Egyetem utca 10., Hungary\\
\small \url{barat@renyi.hu} \\
and\\
Zoltán L. Blázsik\thanks{\protect\includegraphics[height=1cm]{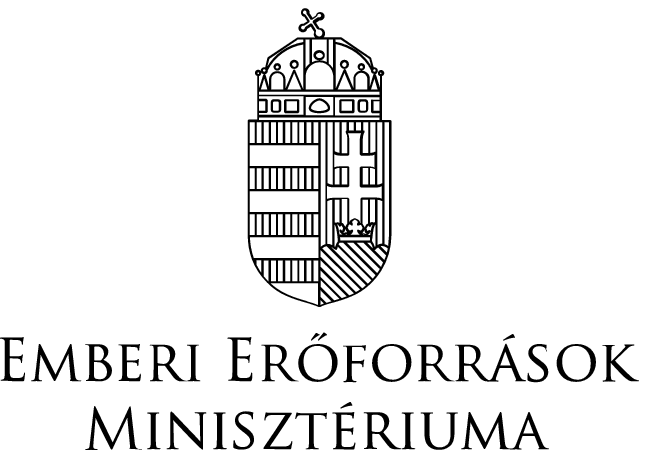} The author was supported by the \'UNKP-22-4-SZTE-480 New National Excellence Program of the Ministry of Human Capacities. The research was supported by the Hungarian National Research, Development and Innovation Office, OTKA grant no. SNN 132625.} \\
\small Alfr\'ed R\'enyi Institute of Mathematics\\
\small MTA--ELTE Geometric and Algebraic Combinatorics Research Group \\
\small SZTE Bolyai Institute \\
\small \url{blazsik@renyi.hu} \\
}
\date{\today}

\begin{document}

\maketitle

\begin{abstract}
In an orientation $O$ of the graph $G$, the edge $e$ is deletable if and only if $O-e$ is strongly connected.
For a $3$-edge-connected graph $G$, H\"orsch and Szigeti defined the Frank number as the minimum $k$ for which $G$ admits $k$ orientations such that every edge $e$ of $G$ is deletable in at least one of the $k$ orientations.
They conjectured the Frank number is at most $3$ for every $3$-edge-connected graph $G$.
They proved the Petersen graph has Frank number $3$, but this was the only example with this property.
We show an infinite class of graphs having Frank number $3$.
H\"orsch and Szigeti showed every $3$-edge-colorable $3$-edge-connected graph has Frank number at most $3$.
It is tempting to consider non-$3$-edge-colorable graphs as candidates for having Frank number greater than $2$.
Snarks are sometimes a good source of finding critical examples or counterexamples.
One might suspect various snarks should have Frank number $3$.
However, we prove several candidate infinite classes of snarks have Frank number $2$.
As well as the generalized Petersen Graphs $GP(2s+1,s)$.
We formulate numerous conjectures inspired by our experience.
\end{abstract}

\section{Introduction}
The graphs in this paper are finite and without loops or multiple edges. We recommend the book by Bondy and Murty \cite{b&m} for the concepts and notations used here.

A graph $G$ is defined by its vertex set $V$ and edge set $E$.
An {\it orientation} of $G$ is a directed graph $D= (V,A)$ such that each edge $uv\in E$ is replaced by exactly one of the arcs $(u,v)$ or $(v,u)$.

A {\it circuit} is a directed cycle.
A graph is {\it cubic} if every vertex has degree $3$.
A {\it chord} of a cycle or circuit $v_1,\dots,v_k$ is an edge connecting two non-consecutive vertices.

A graph is $3$-edge-connected if and only if the removal of any two edges leaves a connected graph.

An oriented graph is {\it strongly connected} if and only if selecting two arbitrary vertices $x$ and $y$, there is a directed $(x,y)$-path.
An orientation of $G$ is {\it k-arc-connected} if and only if the removal of any $k-1$ arcs leaves a strongly connected oriented graph.

\begin{theorem}[Robbins]
A graph has a strongly connected orientation if and only if it is $2$-edge-connected.
\end{theorem}

The following theorem is a fundamental result in the theory of directed graphs \cite{n-w}.

\begin{theorem}[Nash-Williams] \label{t:NW}
A graph has a $k$-arc-connected orientation if and only if it is $2k$-edge-connected.
\end{theorem}

This opens the question for orientations of $3$-edge-connected graphs.
This was the motivation for H\"orsch and Szigeti \cite{szi} for the following concepts.
In an orientation $O$ of $G$ the edge $e$ is {\it deletable} if and only if $O-e$ is strongly connected.
For a $3$-edge-connected graph $G$, H\"orsch and Szigeti defined the {\it Frank number} $F(G)$ as the minimum $k$ for which $G$ admits $k$ orientations such that every edge $e$ of $G$ is deletable in at least one of the $k$ orientations.
H\"orsch and Szigeti \cite{szi} showed that any $3$-edge-connected graph $G$ satisfies $F(G)\le 7$ improving on an earlier result.

They also showed any $3$-edge-colorable $G$ has Frank number at most $3$, and the Petersen graph has Frank number $3$.
These results made us think probably some other non-$3$-edge-colorable graphs might have Frank number larger than 2.
Snarks are $4$-edge-chromatic cubic graphs and usually their girth is at least $5$.
The Petersen graph is the smallest snark.
The next smallest are the Blanu\v sa snarks.
They have Frank number 2.
We also studied an infinite snark family. In Section~\ref{snark}, we show that each Flower snark has Frank number 2.

Some crucial properties of the Petersen graph can be generalized to the so called Generalized Petersen graphs $GP(2s+1,s)$.
One might hope to find a graph among those, which has Frank number $3$.
However, we prove in Section \ref{sec:genpet} that $F(GP(2s+1,s))=2$ for $s\ge 3$.

In Section~\ref{sec:pre}, we consider a few typical infinite families of 3-edge-connected graphs and determine their Frank number, that turns out to be 2.
We introduce a useful tool, that help us checking the Frank number.
We also mention several facts, which we later use implicitly.

These results lead to the question whether there are any graphs with Frank number greater than 2 besides the Petersen graph. As our main result, we construct infinitely many graphs with Frank number 3 in Section~\ref{sec:lcm}. We show an operation, which preserves the Frank number and the edge-connectivity of $3$-edge-connected graphs, and produces a cubic graph from a cubic graph.
A graph $H$ is a {\it truncation} of a cubic graph $G$ if a vertex $v$ of $G$ is replaced by a triangle $v_1,v_2,v_3$ such that each neighbour of $v$ is adjacent to one of $v_1,v_2,v_3$ so that $H$ remains cubic. Truncation was probably first used in connection with Hamiltonian cycles of polyhedra. In Section \ref{sec:lcm}, we introduce the \emph{local cubic modification}, which generalizes truncation to vertices of larger degree.

\begin{theorem} \label{t:Frank3}
There are infinitely many cubic graphs  $G$ such that $F(G)=3$.
They can be constructed from the Petersen graph by successive truncations.
\end{theorem}

For instance, the first truncation of the Petersen graph is the Tietze graph. We prove exhaustively that indeed the Petersen graph is the only cubic $3$-edge-connected graph on at most 10 vertices having Frank number $3$. Inspired by the proofs of H\"orsch and Szigeti, we can show the following.

\begin{theorem} \label{triangle-free}
Let $G$ denote a $3$-edge-connected graph such that $F(G)\ge 3$. Then there exists a cubic triangle-free graph $H^{\ast}$ such that $F(H^{\ast})\ge F(G)\ge 3$.
\end{theorem}

In Section~\ref{sec:small} we show that the Petersen graph is the only one with Frank number greater than 2
among cubic 3-edge-connected graphs on at most 10 vertices.

\section{Preliminaries} \label{sec:pre}

If $O$ is an orientation of $G$, then let $-O$ be the orientation, which we get by reversing every arc in $O$.

\begin{fact} \label{reverse}
The set of deletable edges is the same for $O$ and $-O$.
\end{fact}

We routinely have to check if an edge is deletable.
The following observation shows one way to do that.

\begin{proposition}\label{lem:str_con_check}
Let $G$ be an arbitrary $2$-edge-connected graph, and $e=uv \in E(G)$. Suppose that $O$ is a strongly connected orientation of $G$ such that the arc corresponding to $e$ goes from $u$ to $v$. The orientation $O-e$, which we get by deleting the arc $(u,v)$ from $O$, is strongly connected if and only if there exists a directed path in $O-e$ from $u$ to $v$.
\end{proposition}

\begin{proof}
If there is no $(u,v)$-path in $O-e$, then $O-e$ is not strongly connected by definition.

If there is a $(u,v)$-path $P$ in $O-e$, then in any $(x,y)$-path of $O$, which uses the arc $(u,v)$, we replace $(u,v)$ by $P$.
Since $O$ was strongly connected, we now find an $(x,y)$-walk in $O-e$ for any pair $x$ and $y$. Therefore $O-e$ is strongly connected.
\end{proof}

Let us remark that Proposition \ref{lem:str_con_check} is true even if the edge $e$ is contained in an edge cut $C$ of size 2. In this case, $O-e$ cannot admit a strongly connected orientation and we can deduce this by showing that there are no directed paths from $u$ to $v$. Suppose to the contrary $O-e$ contains a directed path from $u$ to $v$. Consequently, $C$ must be a directed cut contradicting that $O$ is a strongly connected orientation.

\begin{fact}
Suppose $G$ is a graph and its strongly connected orientations $O_1, O_2,\dots,O_k$ show $F(G)=k$. By the strong connectivity, there is no sink or source of degree $3$ in $O_i$, for any $i\in\{1,2,\dots,k\}$.
\end{fact}

In an oriented graph, a vertex $x$ of total degree 3 is red, if there are precisely two arcs leaving $x$, similarly green, if there are precisely two arcs entering $x$. The following observation gives a necessary but not sufficient condition on the deletability of an arc in a cubic graph.

\begin{fact}
If $G$ is a cubic graph and $O$ is a strongly connected orientation of $G$, then an arc $e=(u,v)$ can be deletable only if $u$ is red and $v$ is green.
\end{fact}

By Proposition \ref{lem:str_con_check}, the deletability of the arc $(u,v)$ is equivalent to the existence of a directed path from $u$ to $v$ in $O-e$. Therefore $u$ must have outdegree exactly 2, and $v$ must have indegree exactly 2. However, the example in Figure ~\ref{fig:rg} shows that these degree conditions are insufficient. If there exists an edge cut containing $e$ such that every arc except $e$ are going in the same direction, then after deleting $e$, this edge cut becomes a directed cut, hence no directed $(u,v)$-path exist anymore regardless of the in- and outdegree of $u$ and $v$.

\begin{figure}[!h]
    \centering
    \includegraphics[width=0.6\textwidth]{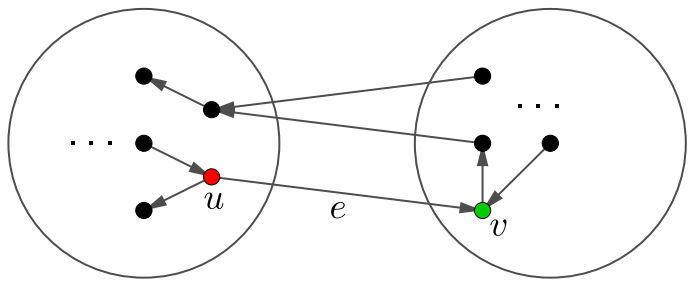}
    \caption{The arc $e=(u,v)$ is not deletable despite the fact that $u$ is red, and $v$ is green}
    \label{fig:rg}
\end{figure}

We use the following observation repeatedly. If $O$ is a strongly connected orientation of a $2$-edge-connected graph and $C$ is a circuit of $O$, then every chord of $C$ is deletable regardless of its orientation. Thus if $O$ contains a Hamiltonian circuit $C$, then every arc of $O-C$ is deletable.

\subsection{Three elementary classes}

We use the following three results later in Section~\ref{sec:small}. A typical $3$-edge-connected graph family is the so called wheel graphs.
For a positive integer $n\ge 3$, the wheel $W_n$ consists of a hub vertex $v_0$ and $n$ other vertices forming a cycle such that $v_0$ is adjacent to all other vertices forming the spoke edges.
Notice that $W_3$ is the complete graph on 4 vertices.

\begin{lemma}
For every positive integer $n\ge 3$, the wheel $W_n$ has Frank number $2$.
\end{lemma}

\begin{proof}
Let $n$ be even.
We give the first orientation of the edges of $W_n$ as follows.
We orient the edges of the outer $n$-cycle to get a circuit.
We alternately orient the spoke edges to and from $v_0$ starting at $v_1$.
Now every arc leaving a red vertex is deletable.
They appear alternately both on the outer circuit and on the spokes.

Now it is easy to give a second orientation, where the remaining arcs are deletable.
We simply reverse every spoke.
These two orientations of $W_n$ show that the Frank number is 2.

Let now $n$ be odd.
We give the first orientation of the edges of $W_n$ as follows.
We orient the edges from $v_1$ to $v_n$ to get a directed path.
However, we orient the last edge from $v_1$ to $v_n$.
We alternately orient the spoke edges to and from $v_0$ except that $v_0v_1$ and $v_0v_2$ are both going outwards.
Now every arc from a red vertex to a green vertex is deletable.
They appear alternately on the outer cycle such that $(v_1,v_2)$ and $(v_1,v_n)$ are both deletable.
Every spoke edge is deletable except $(v_n,v_0)$ and $(v_0,v_1)$.

In the second orientation, we orient the outer cycle $v_n,v_{n-1},\dots, v_1$ to get a circuit.
We alternately orient the spoke edges to and from $v_0$ such that $v_1v_0$ and $v_nv_0$ are both going to $v_0$.
Now the odd-indexed vertices are red including both $v_1$ and $v_n$.
Therefore, all non-deletable arcs of the first orientation are deletable in the second.
\end{proof}

For an even integer $n\ge 4$, let the Möbius ladder $M_n$ be defined as follows.
Let $v_1,\dots,v_n$ be a cycle and we connect each opposite pair, these are edges of form $v_iv_{i+n/2}$.

\begin{lemma}
For every positive even integer $n\ge 4$, the graph $M_n$ has Frank number $2$.
\end{lemma}

\begin{proof}
Let $n/2$ be an odd number.
We give the first orientation of the edges as follows.
We orient the cycle edges consecutively $(v_i,v_{i+1})$ to get a circuit.
This implies every diagonal edge is deletable independent of its orientation.
We orient the diagonal edges alternately.
That is, every vertex $v_i$ with odd index is the tail of an arc, and every even-indexed vertex is the head.
It implies that every second arc of the outer circuit is deletable,
since an arc $(v_i,v_{i+1})$ can be replaced by a directed path with 3 arcs.

Therefore it is immediate to construct the second orientation by reversing the diagonals.
\footnote{It is visually a rotation of the first orientation.}
Now every edge is deletable in at least one of the two orientations.

Let $n/2$ be an even number now.
The first orientation is the same as in the previous (odd) case.
We orient the cycle edges consecutively $(v_i,v_{i+1})$ to get a circuit.
We orient the diagonal edges alternately until the arc with head $v_{n/2}$.
Thereby, all diagonal edges are deletable and every second arc of the outer circuit, except that both arcs of the outer circuit incident to $v_n$ and $v_{n/2+1}$ are non-deletable.

Now we have to find a second orientation such that every second arc of the outer circuit is deletable, plus the two extra arcs we listed.
We create the following circuit:\\ $v_1,v_2,v_{n/2+2},v_{n/2+3},v_3,v_4,v_{n/2+4},v_{n/2+5},\dots,
v_{n/2-1},v_{n/2},v_n$, which contains every vertex but $v_{n/2+1}$.
We add $(v_{n/2},v_{n/2+1})$, $(v_{n/2+2},v_{n/2+1})$ and $(v_{n/2+1},v_1)$ thereby making the orientation strongly connected.
Therefore every second edge $(v_2,v_3),\dots,(v_{n/2-2},v_{n/2-1})$ and $(v_{n/2+3},v_{n/2+4}),\dots,(v_{n-1},v_n)$ are deletable independent of their orientation.
The extra arcs, we need are: $(v_{n/2},v_{n/2+1})$, $(v_{n/2+2},v_{n/2+1})$, $(v_{n-1},v_n)$ and $(v_n,v_1)$. Therefore, it is necessary to orient the edge $v_{n-1}v_n$ away from $v_n$.
This is routine to check the existence of the necessary directed paths.
Every edge is deletable in at least one of the two given orientations of $M_n$.
\begin{figure}[!h]
    \centering
    \includegraphics[width=0.8\textwidth]{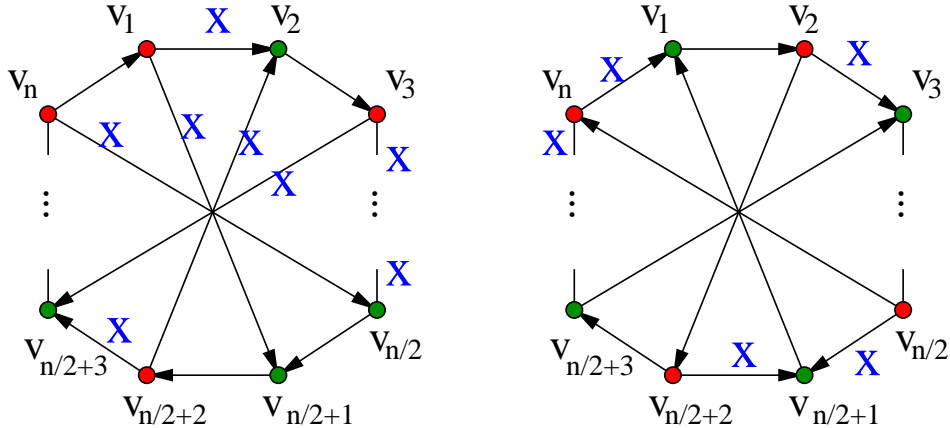}
    \caption{Two appropriate orientations of a Möbius ladder.}
    \label{fig:lcm}
\end{figure}
\end{proof}

\begin{lemma}
For every $k$, the prism $P_k=C_k\times K_2$ has Frank number $2$, where $k\ge 3$.
\end{lemma}

As prisms are almost identical to Möbius ladders, the next proof is similar to the previous one.

\begin{proof}
Let $k$ be even.
We denote the outer cycle by $v_1,\dots v_k$ and the inner cycle by $u_1,\dots, u_k$ and the spoke edges by $u_iv_i$ for every $1\le i\le k$.
We give the first orientation of the edges as follows.
We orient the outer cycle edges consecutively $(v_i,v_{i+1})$ to get a circuit.
We orient the inner cycle edges consecutively backwards $(u_{i+1},u_i)$ to get a circuit.
We orient the spoke edges alternately.
That is, every vertex $v_i$ with odd index is the tail of a spoke arc, and every even-indexed vertex $v_i$ is the head.
This implies every spoke edge is deletable by Proposition~\ref{lem:str_con_check} finding a directed path using five edges.
By the same manner, we find that every arc from a red vertex to a green vertex is deletable on the outer or the inner circuit.
Since these arcs appear alternately, we get the second orientation by reversing the spokes\footnote{Visually rotating every arc by one.}.
We find that every edge is deletable in at least one of the two orientations.

Let $k$ be odd.
We denote the outer cycle by $v_1,\dots v_k$ and the inner cycle by $u_1,\dots, u_k$ and the spoke edges by $u_iv_i$ for every $1\le i\le k$.
We give the first orientation of the edges as follows.
We orient the outer cycle edges consecutively $(v_i,v_{i+1})$ to get a circuit.
We orient the inner cycle edges consecutively backwards $(u_{i+1},u_i)$ to get a circuit.
We orient the spoke edges alternately.
That is, every vertex $v_i$ with odd index is the tail of a spoke arc, and every even-indexed vertex $v_i$ is the head.
This implies every spoke edge is deletable by Proposition~\ref{lem:str_con_check} finding a directed path using the inner and outer circuit.
By the same manner, we find that every arc from a red vertex to a green vertex is deletable on the outer or the inner circuit.

We give the second orientation of the edges as follows.
We orient the outer cycle edges consecutively forward as $(v_i,v_{i+1})$ except $(v_2,v_1)$.
We orient the inner cycle edges consecutively backwards as $(u_{i+1},u_i)$ except $(u_1,u_2)$.
We reverse the spoke edges.
That is, every vertex $v_i$ with odd index is the tail of a spoke arc, and every even-indexed vertex $v_i$ is the head.
In particular, there are two circuits we use in the next part of the proof:
$C=v_1,u_1,u_2,v_2,v_3,\dots, v_k$ and $C'=u_1,u_k,\dots,u_2,v_2,v_1$.
We use Proposition~\ref{lem:str_con_check} again to show deletable edges.
The four arcs $(v_k,v_1), (v_2,v_1), (u_1,u_k), (u_1,u_2)$ are special.
We spell out one of them.
E.g. $(v_k,v_1)$ can be replaced by the path $v_k,u_k,\dots,u_2,v_2,v_1$.
For all other arcs from a red vertex to a green vertex, we use either $C$ or $C'$ to find the alternative directed path.
E.g. $(v_3,v_4)$ can be replaced by $(v_3,u_3),C',(u_4,v_4)$.
\end{proof}

\section{Local cubic modification} \label{sec:lcm}

Hörsch and Szigeti~\cite{szi} introduced the notion of \emph{cubic extensions} of a graph with minimum degree at least 3 in Subsection 2.3. It is a global modification, which replaces every vertex $v$ of degree at least 4 with a cycle of size $deg(v)$, leave the vertices of degree 3 intact, and substitute every edge with an edge between the corresponding objects in such a way that this not necessarily unique graph is cubic.

In contrast to that, we use the following local operation on a graph $G$ of minimum degree at least 3. For $d\ge 3$, let $v$ be a vertex of degree $d$, and let the neighbours of $v$ be $x_1,\dots , x_d$. We remove $v$ and replace each edge $vx_i$ by an edge $v_jx_i$ and add a cycle $C_v$ on $v_1,\dots v_d$ (see Figure \ref{fig:lcm}) so that each of the new vertices has exactly one neighbour from $x_1,x_2,\dots,x_d$. The resulting graph $G_v$ is a {\it local cubic modification} of $G$ at $v$. Let us remark that $G_v$ is not necessarily unique, it depends on the chosen perfect matching between $\{x_1,x_2,\dots,x_d\}$ and $\{v_1,v_2,\dots,v_d\}$. Note that for $d=3$ the truncation is a special local cubic modification.

\begin{figure}[!h]
    \centering
    \includegraphics[width=0.7\textwidth]{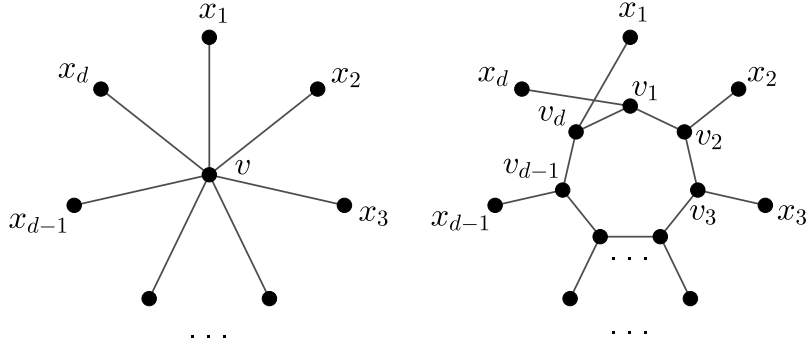}
    \caption{A local cubic modification at $v$}
    \label{fig:lcm}
\end{figure}

Let us emphasize that at this point it may happen that after performing a local cubic modification the edge-connectivity decreases (see Figure \ref{fig:ecdecrease}).

\begin{figure}[!h]
    \centering
    \includegraphics[width=0.8\textwidth]{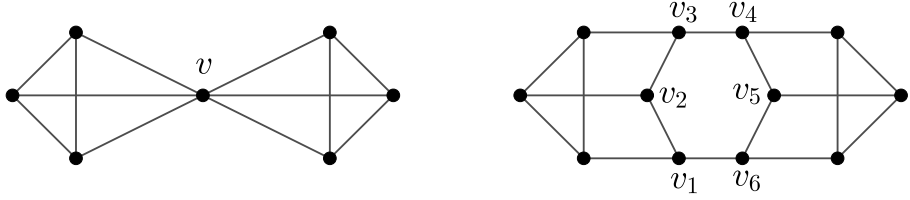}
    \caption{The edge-connectivity may decrease by performing a local cubic modification at $v$}
    \label{fig:ecdecrease}
\end{figure}

In this paper, we are interested in $3$-edge-connected graphs and for such graphs we show that there exists a local cubic modification, which remains $3$-edge-connected. Let $G$ denote an arbitrary $3$-edge-connected graph and let $v$ be an arbitrary vertex of $G$. Denote by $G_v(M)$ the local cubic modification of $G$ at $v$ with the chosen perfect matching $M$.

If $deg(v)=d\le 5$, then suppose to the contrary there is an edge cut $C=\{e,f\}$ of size 2 in $G_v(M)$. Let $A$ and $B$ denote different non-empty connected components of $G_v(M)-C$.

Suppose $A=\{v_{j_1},v_{j_2},\dots v_{j_\ell}\}\subseteq \{v_1,v_2,\dots,v_d\}$. Now $|A|\le 2$ since $e$ and $f$ must contain the edges of $M$ incident to $v_{j_i}$ for all $i\in\{1,2,\dots,\ell\}$. On the other hand, $e$ and $f$ must belong to the new cycle $C_v$ otherwise all of the vertices $\{v_1,v_2,\dots,v_d\}$ would be in the same connected component which cannot happen since in that case $|A|=d\ge 3$. This is a contradiction, hence both $A$ and $B$ have a vertex outside of $\{v_1,v_2,\dots,v_d\}$.

Assume $A,B \nsubseteq \{v_1,v_2,\dots,v_d\}$, and choose two vertices $y$ and $z$ from $A\setminus \{v_1,v_2,\dots,v_d\}$ and $B\setminus \{v_1,v_2,\dots,v_d\}$ respectively. Since $G$ is 3-edge-connected, there are at least 3 edge-disjoint paths between $y$ and $z$ by Menger's theorem. Therefore there exists a path $P^{yz}$ disjoint from $\{e,f\}$ in $G$. If $P^{yz}$ does not go through $v$, then the same path exists in $G_v(M)$, a contradiction. On the other hand, if $P^{yz}$ passes through $v$, then we can complete $P^{yz}$ to $P^{yz}_v$ in $G_v(M)$ by connecting the corresponding $v_i$ and $v_j$ ($i$ and $j$ are not necessarily distinct) vertices using the cycle $C_v$.
This completion can be done unless both $e$ and $f$ are edges of $C_v$.
In that latter case, $v$ would be a cut vertex in $G$,
which leads to contradiction since $deg(v)\le 5$ and we could find an edge cut of size at most 2 in $G$. Moreover, this argument shows that $G_v(M)$ is 3-edge-connected unless $v$ is a cut vertex of $G$.

What can we do if $v$ is indeed a cut vertex of $G$?

Denote by $K_1,K_2,\dots,K_k$ the non-empty connected components of $G-v$, where $k\ge2$. Observe that both edges of the edge cut $C$ must belong to the cycle $C_v$ because otherwise the vertices $v_1,v_2,\dots,v_d$ are in the same connected component, thus $C$ leads to an edge-cut of size at most 2 in $G$, a contradiction. Hence $G_v(M)-C$ is not connected if and only if $M$ is chosen such that for all $i\in\{1,2,\dots,k\}$ the vertices of $K_i\cap N_G(v)$ are connected to vertices from the same arc of the two arcs of $C_v-C$. We claim that $M$ can be chosen so that for any pair of edges of $C_v$ the previous condition fails.

Since $G$ is $3$-edge-connected, $|K_i \cap N_G(v)|\ge 3$ for any $i\in \{1,2,\dots,k\}$. One can interpret the choice of $M$ as an assignment of the vertices of $C_v$ to the corresponding connected components of $G-v$. The assignment shown in Figure \ref{fig:joM} fulfills that no matter how we select two edges $\{e',f'\}$ of $C_v$ there exists an $i\in\{1,2,\dots,k\}$ such that the vertices assigned to $K_i$ intersects both arcs of $C_v-\{e',f'\}$.

\begin{figure}[!h]
    \centering
    \includegraphics[width=0.4\textwidth]{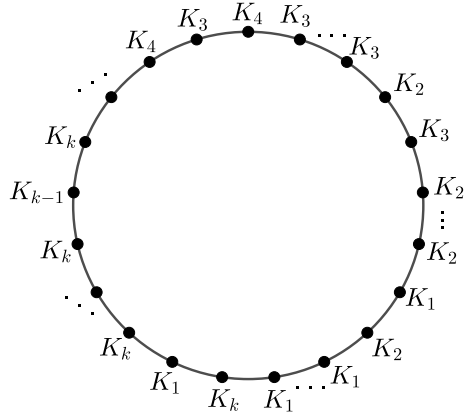}
    \caption{This assignment shows that $M$ can be chosen such that $G_v(M)$ remains $3$-edge-connected even if $v$ is a cut vertex.}
    \label{fig:joM}
\end{figure}

Therefore in the rest of the paper, we assume the local cubic modification $G_v$ of the $3$-edge-connected graph $G$ at vertex $v$ is always $3$-edge-connected. However, it is true that every cubic extension of a graph can be realized as a series of local cubic modifications, and in the other direction if we perform a series of local cubic modifications of a graph at all vertices of degree at least 4, then we get a cubic extension. Consequently, the previous observation means that one can find a $3$-edge-connected cubic extension of a $3$-edge-connected graph even if there are cut vertices.

The following general observation plays a key role in the next proofs when applied to local cubic modification.

\begin{fact}\label{fact:supress2}
Let $G_v$ be the local cubic modification of $G$ at $v$, and an orientation $O_v$ is given such that there exists a directed $(y,z)$-path $P^{yz}_v$ in $O_v$ for $\{y,z\}\nsubseteq \{v_1,v_2,\dots,v_d\}$. Now a directed $(y,z)$-path also exists for the inherited orientation $O$ of $G$ for the corresponding $(y,z)$ pair.
\end{fact}

Now we are ready to show that a local cubic modification cannot decrease the Frank number.
Moreover, if the vertex $v$ has degree 3, then it cannot increase either. Hence in that case, the Frank number remains the same.

\begin{lemma} \label{lem:lcm}
Let $G$ be a $3$-edge-connected graph.
If $G_v$ is a local cubic modification of $G$ at $v$, then $F(G_v)\ge F(G)$.
\end{lemma}

\begin{proof}
Suppose to the contrary that $F(G_v)=k<F(G)$ witnessed by the strongly connected orientations $O^v_1,\dots,O^v_k$.
Let $O_1,\dots , O_k$ be the orientations of $G$, which coincide with $O^v_1,\dots,O^v_{k}$ on identical edges.
Also let the direction of $v_jx_i$ be copied to $vx_i$ in each orientation.
Since each $O^v_j$ was strongly connected, for any pair of vertices $y,z$ there exists a directed path between them in both directions. By Fact~\ref{fact:supress2}, we can deduce that $O_j$ also has the same property hence it is strongly connected.

We claim each edge $e=yz$ of $G$ is deletable in at least one orientation. Let $O^v_j$ be the orientation of $G_v$, where $e$ with the appropriate orientation (say $(y,z)$)  was deletable. We know that $O^v_j$ is strongly connected and contains a directed $(y,z)$-path $P^v_{yz}$ in $O^v_j-\{e\}$. Consequently, similarly to the proof of Fact \ref{fact:supress2}, $O_j-\{e\}$ contains a directed $(y,z)$-path since in $P^v_{yz}$ we can contract the part between the first and last appearance of some $v_i$ for an appropriate $i$. Therefore $e$ is deletable in $O_j$ by Proposition \ref{lem:str_con_check}.
\end{proof}

\begin{corollary} \label{cubicext}
Let $G$ be a $3$-edge-connected graph. There exists a cubic extension $H$ of $G$, which is $3$-edge-connected and $F(H)\ge F(G)$.
\end{corollary}

By Lemma \ref{lem:lcm}, we can create an infinite family $\mathcal{G}$ of cubic graphs with $F(G)\ge 3$ for any $G\in \mathcal{G}$ starting from the Petersen graph in the following way.
Hörsch and Szigeti \cite{szi} showed the Petersen graph has Frank number 3.
Pick an arbitrary vertex $v$ of the Petersen graph, and consider the local cubic modification $G_v$ of $G$ at $v$. Since the Petersen graph is cubic and $3$-edge-connected and $G_v$ is $3$-edge-connected as well, hence by Lemma \ref{lem:lcm}, we get $F(G_v)\ge F(G)$. After iterating this local cubic modification procedure with an arbitrary vertex of the always cubic current graph, the Frank number never decreases. Thus we created an infinite family of $3$-edge-connected graphs with Frank number at least 3.

In Theorem \ref{t:Frank3}, we claimed the existence of an infinite family of cubic graphs with Frank number equal to 3. So far we have seen that the Frank number cannot decrease performing a local cubic modification at an arbitrary vertex $v$. In the next Lemma, we show that the Frank number cannot increase if $deg(v)=3$.

\begin{lemma} \label{trunc}
Let $G$ be a $3$-edge-connected graph and $v$ a vertex of degree $3$. If $G_v$ is a local cubic modification of $G$ at $v$, then $F(G_v) \le F(G)$.
\end{lemma}

\begin{proof}
Suppose the orientations $\mathcal{O}=\{O_1, O_2, \dots, O_k\}$ are the witnesses of $F(G)=k$. We create $k$ orientations $\mathcal{O}^v=\{O^v_1, O^v_2,\dots,O^v_k\}$ of $G_v$ to prove $F(G_v)\le k$. Let us focus on the truncated part of $G_v$, we just copy the orientations from the corresponding $O_i$ outside of the modified part.

Since every $O_i$ is a strong orientation, thus the 3-edge-cut formed by edges $\{av,bv,cv\}$ cannot be a directed cut.

By Fact~\ref{reverse}, we might assume that in every orientation $O_i$, exactly two edges leave $v$. For convenience, instead of referring to $a,b,c$ as the concrete neighbours of $v$, let us permute their roles. We may assume that $a$ denotes the tail of the unique arc entering $v$. In Figure~\ref{fig:truncation}, we introduce the four orientations we use later in this proof. Note that the first two orientations become the same if we interchange the roles of $b$ and $c$, and so do the last two orientations. Hence there are essentially two type of extensions which we use on the truncated part of $G_v$.

\begin{figure}[!h]
    \centering
    \includegraphics[width=0.8\textwidth]{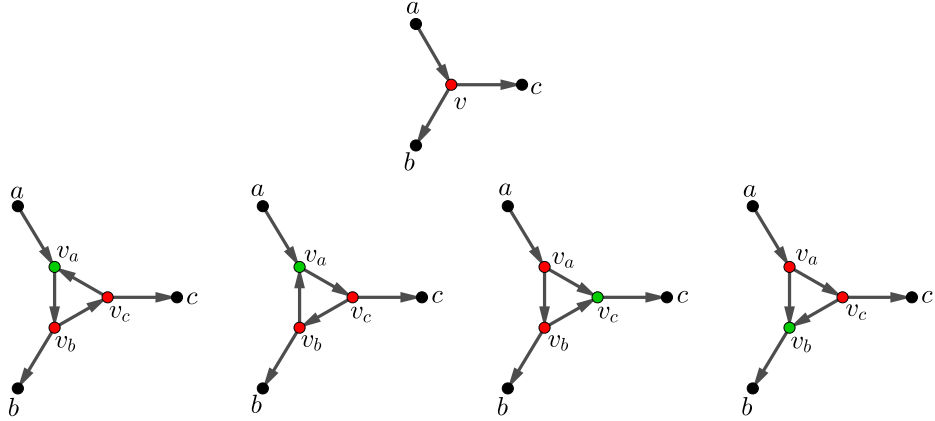}
    \caption{The four orientations we use on the new arcs (essentially two different types)}
    \label{fig:truncation}
\end{figure}

Firstly, observe that no matter which extensions we use from Figure \ref{fig:truncation}, the orientation $O^v_i$ we get is also strongly connected. Indeed, we can enter the triangle $v_a,v_b,v_c$ only from $a$ and we can leave in both directions through $b$ or $c$, hence every directed path of $O_i$ can be extended even if it goes through $v$ in $G$. Moreover, there exists a directed path between any pair of new vertices in $O^v_i$.

An arc of $O_i$ not incident to $v$ is deletable if and only if the same arc is deletable in $O^v_i$. By Proposition \ref{lem:str_con_check}, it is enough to show a directed path between its endpoints in the modified graph as well. As we discussed in the previous paragraph, this can be done and it does not depend on the choice of the orientation of the triangle at the truncated vertex $v$ as long as we use the four orientations above. Therefore for every edge not incident to $v$, there exists an orientation $O^v_i$ of $G_v$ so that the corresponding arc is deletable in $O^v_i$.

Choose a smallest subset $\mathcal{S}=\{O_{j_1},O_{j_2},\dots,O_{j_{\ell}}\}$ of $\mathcal{O}$ such that all of the edges incident to $v$ is deletable in at least one of the orientations in $\mathcal{S}$. Here $1<\ell\le 3$ holds.

If $|\mathcal{S}|=2$, then in at least one of these orientations both arcs leaving $v$ are deletable and in the other orientation the third edge incident to $v$ is not just outgoing but also deletable. In Figure \ref{fig:trun2}, we show how the orientations $\{O^v_{j_1},O^v_{j_2}\}$ look like at the truncated vertex $v$ (remember that the role of $b$ and $c$ are interchangeable). Notice that the blue color and also the X marks (for the black and white versions) indicate which arcs are deletable.

\begin{figure}[!h]
    \centering
    \includegraphics[width=0.5\textwidth]{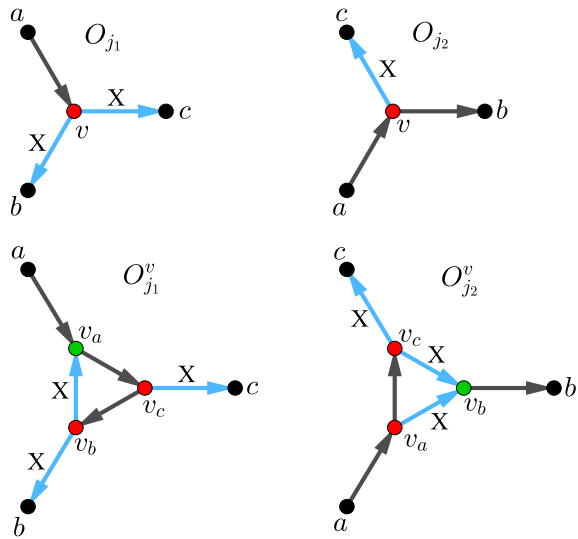}
    \caption{The orientations $\{O^v_{j_1},O^v_{j_2}\}$, if $|\mathcal{S}|=2$}
    \label{fig:trun2}
\end{figure}

Indeed, the arcs of type $(v_x,x)$ are deletable in $O^v_{j_i}$ if and only if $(v,x)$ was deletable in $O_{j_i}$.
The arcs inside the triangle of type $(v_x,v_y)$ are deletable either trivially or because of the fact that $O_{j_i}$ is strongly connected.

If $|\mathcal{S}|=3$, then for each of the edges incident to $v$ there is a unique orientation of $\mathcal{S}$ so that the corresponding arc is deletable. Using any of the last two orientations in Figure \ref{fig:truncation} results in three orientations for which every arc of the triangle is also deletable in at least one of them. Indeed, the arc opposite to the deletable one which leaves $v$ is always deletable by Proposition \ref{lem:str_con_check} since there is a directed path within the triangle.

Naturally, we can use any of the orientations described in Figure \ref{fig:truncation} in any of those orientations of $\mathcal{O}$ which haven't been touched yet. Hence we proved that $F(G_v)\le F(G)$.
\end{proof}

\begin{corollary}
Lemma \ref{lem:lcm} and Lemma \ref{trunc} together implies that if a $3$-edge-connected graph $G$ contains at least one vertex of degree $3$, then by successively performing a local cubic modification at vertices of degree $3$ we get a family of graphs with the same Frank number as $G$. Notice that in each step, the newly introduced vertices have degree $3$.
\end{corollary}

Thus if we start with the Petersen graph, we can build a family of graphs with Frank number exactly 3 concluding the proof of Theorem \ref{t:Frank3}.
However, if a graph $H$ contains a triangle $T$, then we can contract the vertices of $T$ into a new vertex $v_T$ (or in other words identify these vertices) such that $v_T$ is adjacent to the other neighbours of the three vertices of $T$, thus the resulting graph $H/T$ is simple (since $H$ was cubic) and cubic.

What can we say about the relation between the Frank number of $H$ and $H/T$?

Since $H$ is a local cubic modification of $H/T$ at $v_T$, we get $F(H)\ge F(H/T)$ by Lemma \ref{lem:lcm}. On the other hand, Lemma \ref{trunc} yields that $F(H/T)\le F(H)$ since $v_T$ is a vertex of degree 3 in $H/T$. Hence $F(H)=F(H/T)$. Consequently, we can contract triangles starting from $H$ until the resulting graph $H^{\ast}$ is either triangle-free or $H^{\ast}\simeq K_4$ while the Frank number remains the same.
We know that $F(K_4)=2$, and $F(H^{\ast})\ge 2$ if $H^{\ast}$ is a $3$-edge-connected cubic triangle-free graph.

\begin{proof}[Proof of Theorem \ref{triangle-free}]
By Corollary \ref{cubicext}, we can consider the cubic extension $H$ of $G$ for which $F(H)\ge F(G)$. Then after successively contracting triangles the resulting graph $H^{\ast}$ is either triangle-free or it is $K_4$ while $F(H^{\ast})=F(H)$. Since $F(G)\ge3$ thus $H^{\ast}=K_4$ is a contradiction, hence we get a $3$-edge-connected cubic triangle-free graph $H^{\ast}$ such that $F(H^{\ast})\ge F(G) \ge 3$.
\end{proof}

This result may help the computer aided search for other $3$-edge-connected graphs with higher Frank number.

\section{Snarks} \label{snark}

Snarks are bridgeless cubic graphs with chromatic index 4.
The Petersen graph is the smallest such graph.
Hörsch and Szigeti \cite{szi} proved each 3-edge-connected, 3-edge-colorable graph has Frank number at most 3, and the Petersen graph has Frank number 3.
Therefore, we expected to find other examples with Frank number 3 among snarks.

In this section, we investigate the second smallest snarks that are the Blanu\s a snarks and an infinite family of snarks, the so-called flower snarks. For every odd $n\ge 3$ let $J_n$ denote the flower snark on $4n$ vertices. One can construct this graph starting with $n$ copies of stars on 4 vertices with centers $v_1,v_2,\dots,v_n$ and outer vertices denoted by $\{a_i,b_i,c_i\}$ for $1\le i \le n$. Then add an $n$-cycle on the vertices $(a_1,a_2,\dots,a_n)$, and a $2n$-cycle on $(b_1,b_2,\dots,b_n,c_1,c_2,\dots,c_n)$.

It turns out that the Frank number of each of these snarks is 2. The proofs for the two types of snarks are very similar, and we handle them together.

\begin{theorem} \label{t:blaflower}
Both Blanu$\check{s}$\hspace{-1pt}a snarks, and every flower snark has Frank number $2$.
\end{theorem}

\begin{proof} 
Since these snarks are not 4-edge-connected, therefore they do not admit a 2-arc-connected orientation by Theorem~\ref{t:NW}.
Hence their Frank number must be greater than $1$. On the other hand, we show two strongly connected orientations $\{O_1,O_2\}$ of these snarks in Figures \ref{fig:Bla1}, \ref{fig:Bla2}, \ref{fig:flower} that verify that their Frank number is at most 2, which concludes the proof.

\begin{figure}[!h]
    \centering
    \includegraphics[width=\textwidth]{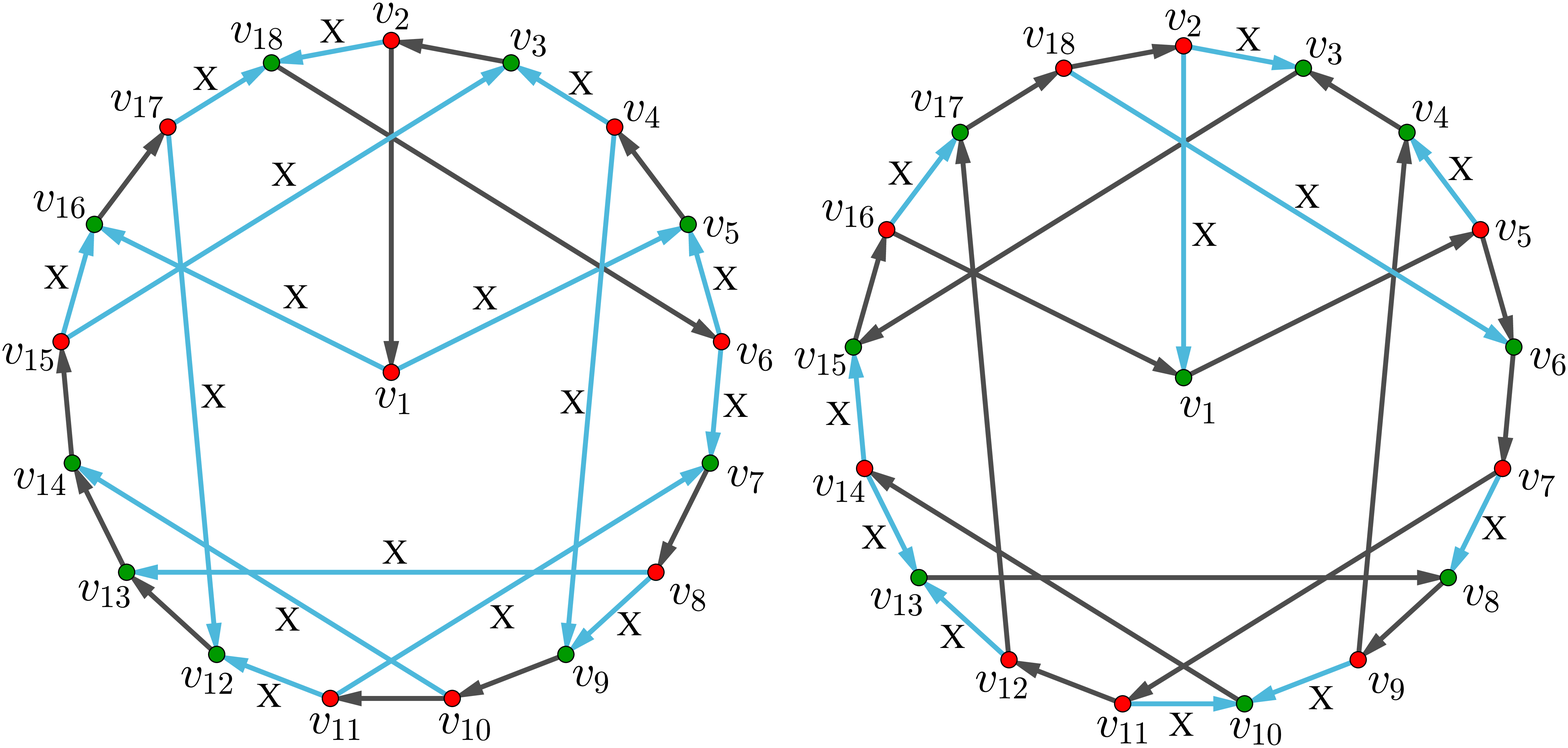}
    \caption{The first Blanu\s a snark has Frank number 2}
    \label{fig:Bla1}
\end{figure}

\begin{figure}[!h]
    \centering
    \includegraphics[width=\textwidth]{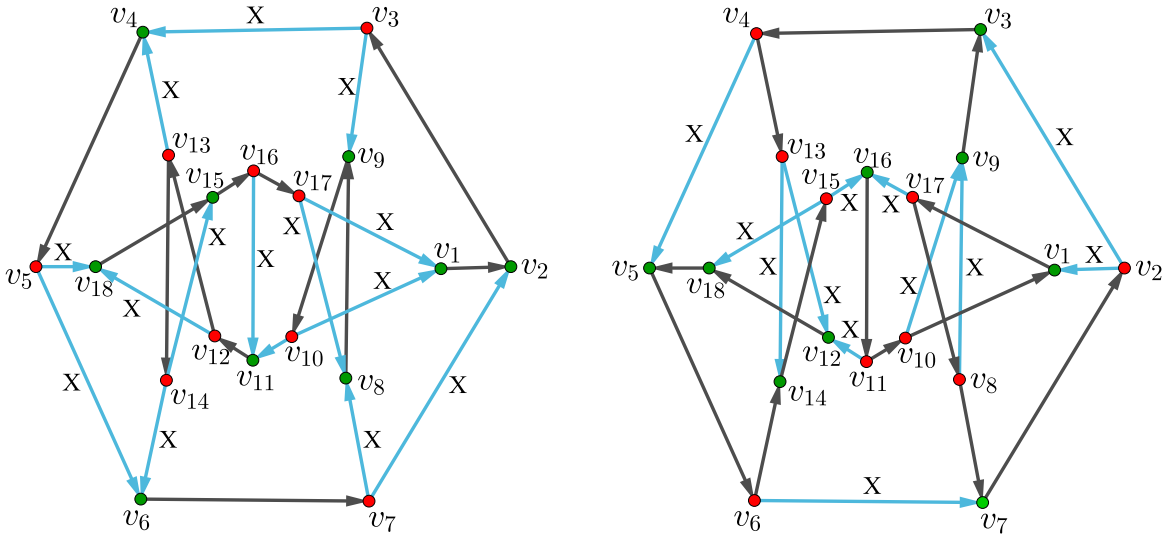}
    \caption{The second Blanu\s a snark has Frank number 2}
    \label{fig:Bla2}
\end{figure}

\begin{figure}[!h]
    \centering
    \includegraphics[width=0.8\textwidth]{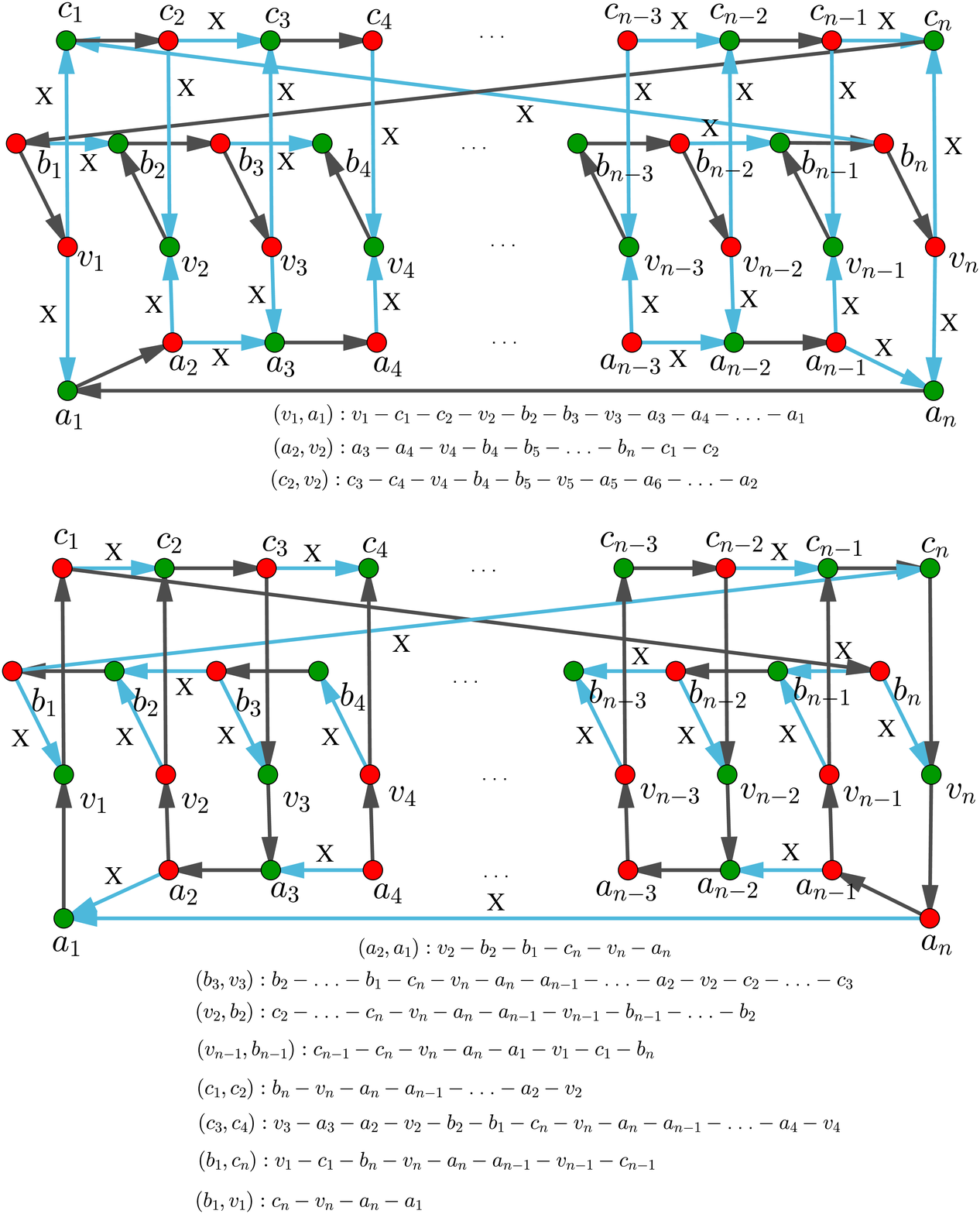}
    \caption{$F(J_n)=2$, for any odd $n$}
    \label{fig:flower}
\end{figure}

The first thing is to check that these orientations are indeed strongly connected. To see this, observe in each orientation, each vertex is covered by a circuit. For any two vertices
there is a chain of intersecting circuits covering these vertices, hence there is a directed path between them in both directions.

To prove that an arc $(u,v)$ is deletable, it is enough to find a directed path from $u$ to $v$ after the deletion of $(u,v)$ by Proposition~\ref{lem:str_con_check}.

In Figures \ref{fig:Bla1}, \ref{fig:Bla2}, \ref{fig:flower} the blue arcs (also marked by X) indicates the deletable arcs of the corresponding orientations. Some hints are included in Figure \ref{fig:flower} which can be generalized for an arbitrary flower snark $J_n$. However, for the two Blanu\s a snarks, there is no general rule (other than using the still intact circuits) for deciding whether an arc is deletable or not, one should manually check them. But finding the appropriate directed path after the deletion is usually straightforward due to the small degrees of the vertices.
\end{proof}

\section{Small graphs}\label{sec:small}

Hörsch and Szigeti \cite{szi} showed the Petersen graph has Frank number 3.
We complement their result and show that any other $3$-edge-connected cubic graph on at most 10 vertices has Frank number at most 2.
We used the House of Graphs \cite{hog} and nauty \cite{nauty} to determine all the candidates.
There are two such graphs on 6 vertices.
One of them is $M_6$, and the other one is the $3$-prism.
Both of them have Frank number 2.

There are four $3$-edge-connected, cubic graphs on 8 vertices.
The cube, $M_8$, the $3$-prism with a handle, and one more.
Figures~\ref{n8_1},~\ref{n8_2} indicate the orientations, which show that the Frank number is 2 for each of these graphs.

\begin{figure}[!ht]
    \centering
    \includegraphics[width=0.6\linewidth]{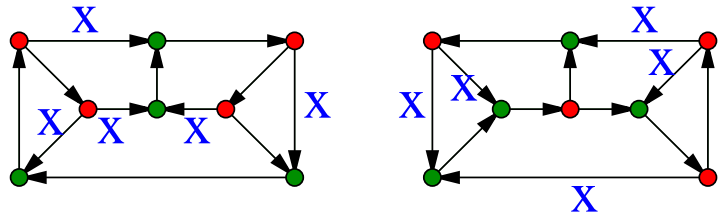}
    \caption{The two orientations of the $3$-prism with a handle}
    \label{n8_1}
\end{figure}

\begin{figure}[!ht]
    \centering
    \includegraphics[width=0.9\linewidth]{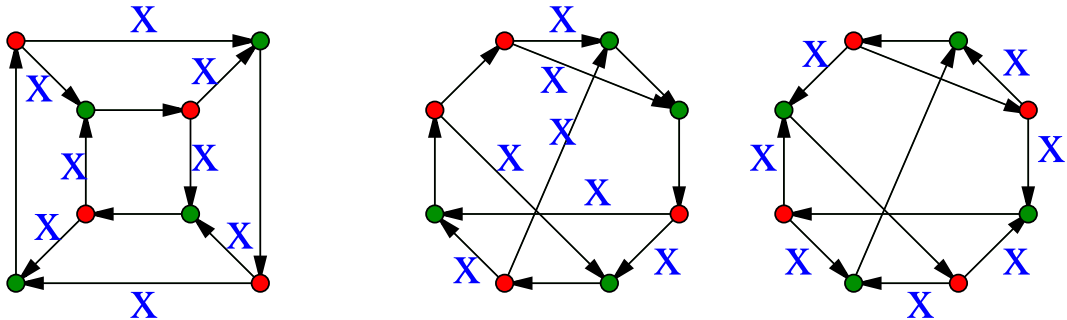}
    \caption{Orientations showing Frank number 2}
    \label{n8_2}
\end{figure}

There are fourteen $3$-edge-connected cubic graphs on 10 vertices, see Figure~\ref{n10}.
Some of them can be constructed from $K_4$ by consecutive truncations.
Since $K_4$ has Frank number 2, we deduce $G_5,G_{11},G_{12}$ have also Frank number 2.
We can obtain $G_{13}$ from $M_6$, and $G_8$ from the cube by consecutive truncations.
Our previous results on wheels and prisms imply $G_2$ and $G_3$ have Frank number 2.
We obtain $G_4,G_7$ from a wheel by truncation.
The Petersen graph, $G_{10}$ is the only graph with Frank number 3.

For the remaining graphs $G_1,G_9,G_{14}$, we always direct the edges of the Hamiltonian cycle as a circuit $v_0,v_1,\dots,v_9$, thereby making the 5 chords deletable.

For $G_1$, in the first orientation, we fix the direction of the long diagonal and we orient the other diagonals to get red and green vertices alternately on the circuit. Now each diagonal is deletable and every second edge of the Hamiltonian cycle. Therefore, by reversing the diagonals, we get an orientation, where the remaining arcs are deletable.

For $G_{14}$, in the first orientation, we add the arcs $(v_9,v_2)$, $(v_1,v_4)$, $(v_3,v_7)$, $(v_5,v_8)$, $(v_6,v_0)$.
We check that the following arcs are deletable: $(v_1,v_2)$, $(v_3,v_4)$, $(v_6,v_7)$, $(v_9,v_0)$.
In the second orientation, we create the following Hamiltonian  circuit: $v_0,v_6,v_5,v_8,v_7,v_3,v_4,v_1,v_2,v_9$.
Thereby, the following arcs are chords and deletable:
$(v_8,v_9)$, $(v_0,v_1)$, $(v_2,v_3)$, $(v_4,v_5)$.
We check that $(v_6,v_5)$ and $(v_8,v_7)$ are also deletable to finish.

For $G_9$, in the first orientation, we add the arcs $(v_0,v_5)$, $(v_7,v_1)$, $(v_2,v_9)$, $(v_6,v_3)$, $(v_4,v_8)$.
We check that the following arcs are deletable:
$(v_0,v_1)$, $(v_2,v_3)$, $(v_4,v_5)$, $(v_7,v_8)$.
In the second orientation, we create the following Hamiltonian  circuit:
$v_0,v_1,v_7,v_6,v_3,v_2,v_9,v_8,v_4,v_5$.
Thereby, the following arcs are chords and deletable:
$(v_2,v_1)$, $(v_4,v_3)$, $(v_5,v_6)$, $(v_9,v_0)$.
We check that $(v_9,v_8)$ and $(v_7,v_6)$ are also deletable to finish.

\begin{figure}[!ht]
    \centering
    \includegraphics[width=\linewidth]{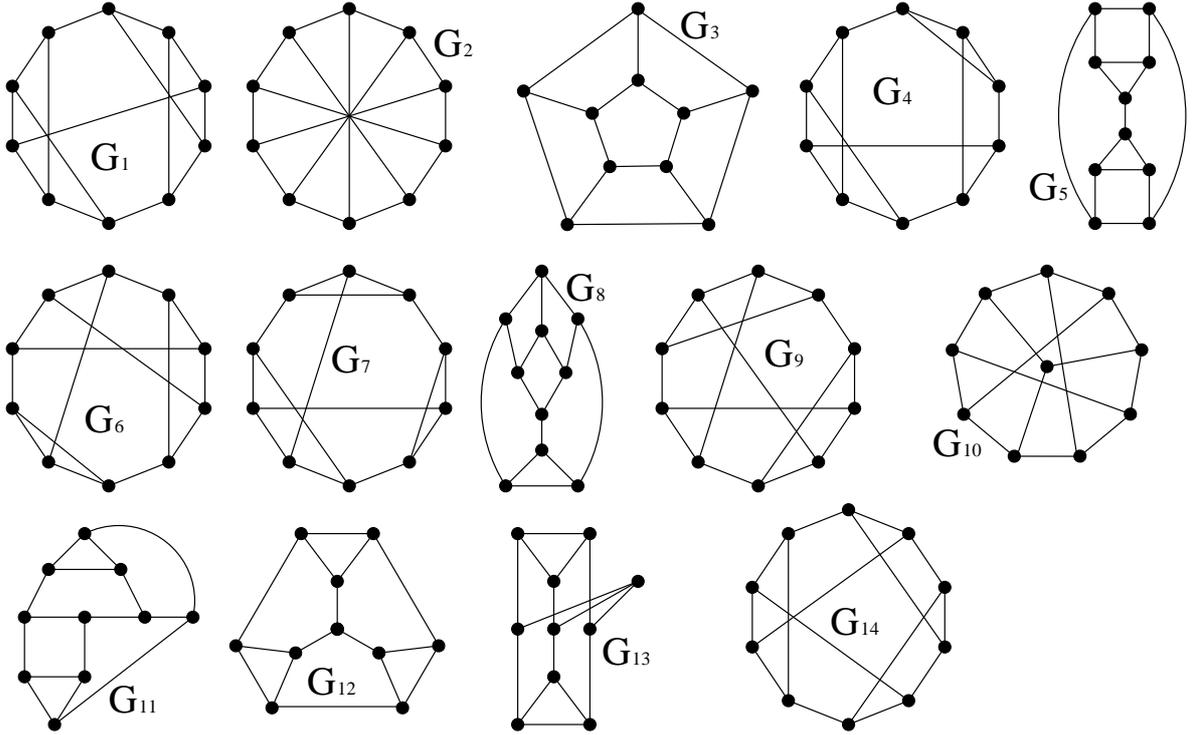}
    \caption{The $3$-edge-connected cubic graphs on 10 vertices}
    \label{n10}
\end{figure}

\section{The Petersen family} \label{sec:genpet}

Hörsch and Szigeti \cite{szi} proved the Petersen graph has Frank number 3.
We give a simple, short, hybrid proof of the fact that the Petersen graph has
Frank number larger than 2.

\begin{proposition}
The Frank number of the Petersen graph is larger than $2$.
\end{proposition}

\begin{proof}
We can determine all non-isomorphic strongly connected orientations of the Petersen graph using nauty \cite{nauty}.
We find there are only 18 such orientations\footnote{See the Appendix}.
In any such orientation, the in- and out-degree of every vertex is at least 1.
If we suppose the Frank number is 2, then we should find two orientations such that
every edge is deletable in at least one of the orientations.
Since there are 15 edges, one of the orientations must contain at least 8 deletable arcs.
We know a deletable arc in a cubic graph connects a red vertex to a green.
There are necessarily 5 red and 5 green vertices in any strongly connected orientation of the Petersen graph.
Since the Petersen graph is non-bipartite, there can be at most 8 deletable arcs in any orientation.
Therefore, we collect the oriented Petersen graphs with 7 or 8 deletable arcs out of the 18 possibilities.
We find that 8 of them satisfy this condition.
We also know there are at most two deletable arcs incident to any vertex.
Therefore, in the first orientation, every vertex must be incident to at least 1 deletable arc.
We find only 4 orientations having this property out of the 8.
These are G(12), G(15), G(17), G(18) in the Appendix. Each of the four graphs has 8 deletable arcs.
It remains to see whether we can combine two sets to cover all edges at least once.
We find the deletable arcs of G(15) and G(17) form a path with 7 edges plus an independent edge, while in G(12) and G(18) we find a 5-path and a 3-path.
This shows we cannot combine two different types or two of the second kind.
It is left to check whether the pairings G(15)-G(15), G(15)-G(17) or G(17)-G(17) works.
However, G(15) and G(17) differs only by the direction of a single arc, which is deletable in both cases.
It remains to check if we can map the vertices of the non-deletable arcs of G(15) to the deletable arcs of the 7-path of G(15) or G(17).
Since this is not the case, we conclude the Frank number of the Petersen graph is larger than 2.
\end{proof}

\begin{figure}[!ht]
    \centering
    \includegraphics[width=0.3\linewidth]{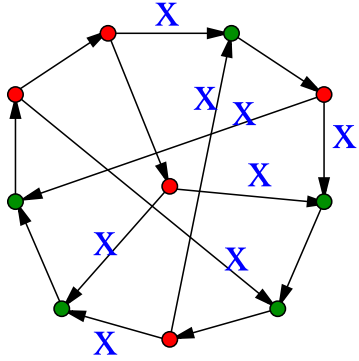}
    \caption{The oriented Petersen graph G(15) with 8 deletable arcs}
    \label{g(15)}
\end{figure}

We investigated the most natural generalized Petersen graphs in the hope of finding another example of a 3-edge-connected graph with Frank number at least 3. As it turned out, the generalized Petersen graph $G(2s+1,s)$ admits two appropriate orientations, consequently its Frank number is 2.

\begin{theorem} \label{t:genpet}
If $GP(2s+1,s)$ denotes the generalized Petersen graph for $s\ge 3$, then $F(GP(2s+1,s))=2$.
\end{theorem}

\begin{proof}
By Theorem~\ref{t:NW}, the graph $GP(2s+1,s)$ does not admit a 2-arc-connected orientation since it is not 4-edge-connected. Thus $F(GP(2s+1,s))>1$.
Note that $GP(2s{+}1,s)$ is not Hamiltonian.
However, it contains a cycle of length $n-1$, where $n=4s+2$ is the number of vertices of $GP(2s+1,s)$. In Figure \ref{fig:genpet} and \ref{fig:genpet2}, we show two orientations of $GP(2s+1,s)$ for $s$ even and odd respectively.
They are not only witnessing that $F(GP(2s+1,s))=2$, but $O_1$ admits the long cycle as a circuit, namely $u_1,u_2,\dots,u_{2s+1},v_{2s+1},v_s,v_{2s},v_{s-1},v_{2s-1},\dots,v_2,v_{s+2},v_1,u_1$. The constructions are very similar depending on the parity of $s$. Orientation $O_1$ is basically the same and $O_2$ should be rotated in opposite directions for the two cases.
The proofs are also very similar.
Therefore, we only give the detailed argument for the even case.
Thus suppose from now on that $s$ is even.

\begin{figure}[!h]
    \centering
    \includegraphics[width=0.9\textwidth]{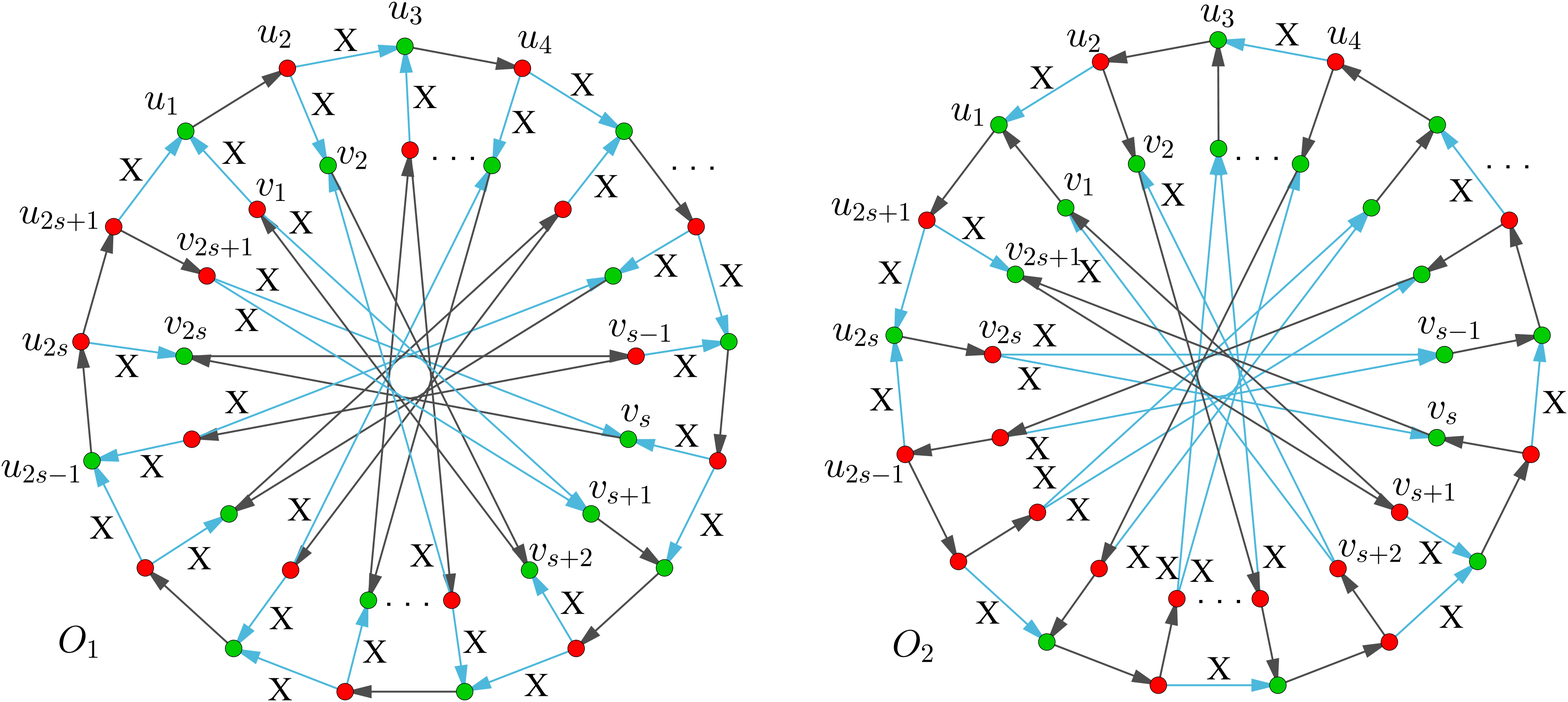}
    \caption{$F(G(2s+1,s))=2$ for $s\ge 3$, $s$ even (illustrated for $s=8$)}
    \label{fig:genpet}
\end{figure}

\begin{figure}[!h]
    \centering
    \includegraphics[width=0.9\textwidth]{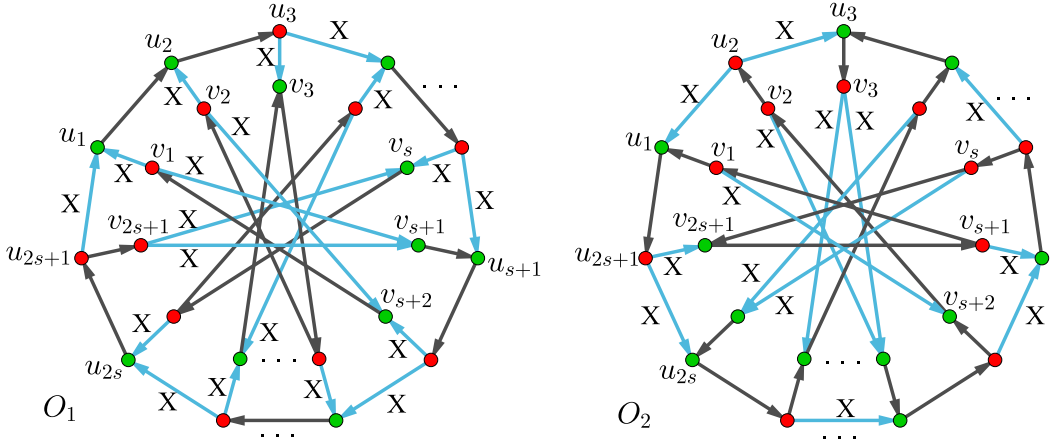}
    \caption{$F(G(2s+1,s))=2$ for $s\ge 3$, $s$ odd (illustrated for $s=5$)}
    \label{fig:genpet2}
\end{figure}

The chords of a circuit are always deletable regardless of their orientation, so we choose them in the following way: for $i\in\{2,\dots,2s\}$ the edge $u_i v_i$ is directed towards $u_i$ if and only if $i$ is odd. Both edges incident to $v_{s+1}$ are directed towards $v_{s+1}$, and the remaining edge $(u_{2s+1},u_1)$ completes the outer cycle to a circuit. Let us remark here that the in-degree of $v_{s+1}$ is just 2, since $s$ is even. This is the obstacle, which we need to solve differently for odd $s$.

We can derive $O_2$ from $O_1$ by keeping the orientation on the edges of type $u_i v_i$, and reversing the orientation of every arc except one $(u_{2s-1},u_{2s})$ in the outer cycle. At this point, the orientation of the edges of the inner cycle is determined using the information which arcs are deletable in $O_1$ (the deletable arcs are blue and marked with X).

First we have to show that these are strong orientations. $O_1$ contains a circuit containing all but one vertex and the cut, formed by the arcs incident to the missing vertex $v_{s+1}$, is not a directed cut. Hence $O_1$ is a strong orientation. $O_2$ has a long directed path $u_{2s-1},u_{2s-2},\dots,u_1,u_{2s+1},u_{2s}$, and there is a circuit $u_{2s},v_{2s},v_{s-1},u_{s-1},u_{s-2},v_{s-2},v_{2s-1},u_{2s-1}$, therefore all vertices $u_i$ belong to the same strongly connected component. One can rotate the following circuit $u_{2s+1},v_{2s+1},v_{s+1},v_1,u_1$ clockwise by 2 positions confirming that $O_2$ is also a strong orientation.

In $O_1$ the arcs between $u_i$ and $v_i$ are deletable (except for $i=s+1$) regardless of their orientation since they are chords of the long circuit. By Proposition \ref{lem:str_con_check}, the blue arcs on the outer part are deletable since there is a directed path between its endvertices using the inner circuit. Similarly, the inner blue arcs are deletable since there is a directed path using the outer circuit. For $O_2$, one can also find the directed paths switching between the use of the inner and outer circuits. Since every arc is deletable in at least one of these orientations, we proved $F(GP(2s+1,s))=2$.
\end{proof}

\section*{Discussion}

We pose the following conjectures, each of which is relaxing the strong conjecture that every $3$-edge-connected graph has Frank number at most 3.

\begin{conjecture}
For every cubic $3$-edge-connected graph $G$, there exists a strongly connected orientation $D$ of $G$ such that for every vertex $v$, there exists an arc $a_v$ incident to $v$ such that $D-a_v$ is strongly connected.
\end{conjecture}

\begin{conjecture}
For every cubic $3$-edge-connected graph $G$, there exist two strongly connected orientations $D_1$ and $D_2$ of $G$ such that for every vertex $v$, there exist two arcs $a^1_v$ and $a^2_v$ incident to $v$ such that $D_1-a^1_v$ and $D_2-a^2_v$ are strongly connected.
\end{conjecture}


\begin{conjecture}
For every cubic $3$-edge-connected graph $G$, there exists a strongly connected orientation $D$ of $G$ such that for at least half of the arcs $D-a$ is strongly connected.
\end{conjecture}

\begin{conjecture}
If a $3$-edge-connected cubic graph $G$ admits a Hamiltonian cycle, then $G$ has Frank number $2$.
\end{conjecture}

This is true up to 12 vertices.

\noindent Permutation snarks on $2k$ vertices are similar to $k$-prisms.
There are two $k$-cycles $v_1,v_2,\dots,v_k$ and $u_1,u_2,\dots,u_k$, and there is a permutation $\phi:[k]\rightarrow [k]$ such that $v_iu_{\phi(i)}$ are the remaining edges.

\begin{question}
Does every permutation snark $G$ have Frank number $2$?
\end{question}

\eject
\newgeometry{top=30mm, bottom=30mm}
\section*{Appendix}
\begin{verbatim}
G(1)         G(2)         G(3)         G(4)         G(5)
0110000000   0110000000	  0110000000   0110000000   0110000000
0000110000   0000110000	  0000110000   0000110000   0000110000
0000001001   0000001001	  0000001001   0000001001   0000001001
1000000100   1000000100   1000000100   1000000100   1000000000
0000001010   0000001000	  0000000010   0000000010   0000001010
0000000001   0000000001	  0000000101   0000000001   0000000100
0000000100   0000000100	  0000100000   0000100000   0000000100
0000010000   0000010000	  0000001000   0000011000   0001000000
0001000000   0001100000	  0001000000   0001000000   0001000001
0000000010   0000000010	  0000000010   0000000010   0000010000

G(6)         G(7)         G(8)         G(9)         G(10)
0110000000   0110000000   0110000000   0110000000   0110000000
0000110000   0000110000   0000110000   0000110000   0000110000
0000001001   0000001001   0000001001   0000001000   0000001000
1000000000   1000000000   1000000000   1000000100   1000000100
0000001010   0000001010   0000001000   0000001010   0000001000
0000000100   0000000001   0000000001   0000000001   0000000001
0000000100   0000000100   0000000100   0000000100   0000000100
0001000000   0001010000   0001010000   0000010000   0000010000
0001000000   0001000000   0001100000   0001000000   0001100000 
0000010010   0000000010   0000000010   0010000010   0010000010

G(11)        G(12)        G(13)        G(14)        G(15)
0110000000   0110000000   0110000000   0110000000   0110000000
0000110000   0000110000   0000110000   0000110000   0000110000
0000001000   0000001000   0000001000   0000001000   0000001000
1000000100   1000000100   1000000010   1000000010   1000000010
0000000010   0000000010   0000001010   0000000010   0000000010
0000000101   0000000001   0000000100   0000000101   0000000100
0000100000   0000100100   0000000100   0000100100   0000100100
0000001000   0000010000   0001000000   0001000000   0001000000
0001000001   0001000000   0000000001   0000000001   0000000001
0010000000   0010000010   0010010000   0010000000   0010010000

G(16)        G(17)        G(18)
0110000000   0110000000   0110000000
0000110000   0000110000   0000110000
0000001000   0000001000   0000001000
1000000000   1000000000   1000000000
0000001000   0000000010   0000000010
0000000001   0000000100   0000000100
0000000100   0000100100   0000100000
0001010000   0001000000   0001001000
0001100000   0001000001   0001000001
0010000010   0010010000   0010010000
\end{verbatim}

\begin{thebibliography}{00}

\bibitem{b&m} J.A. Bondy, U.S.R. Murty.
Graph Theory. {\em Springer-Verlag}, London, XII+663 pages, (2008).

\bibitem{hog} G. Brinkmann, K. Coolsaet, J. Goedgebeur, H. Mélot.
House of Graphs: a database of interesting graphs.
{\em Discrete Applied Mathematics}, 161(1-2):311--314, (2013). Available at \url{http://hog.grinvin.org}


\bibitem{szi}F. Hörsch, Z. Szigeti.
Connectivity of orientations of $3$-edge-connected graphs.
{\em European J. Combin.} {\bf 94} (2021)


\bibitem{nauty}
B.D. McKay and A. Piperno.
Practical graph isomorphism II.
{\em J. Symbolic Comput.} {\bf 60} (2014), 94--112.

\bibitem{n-w} C.St.J.A.  Nash–Williams.
On  orientations,  connectivity,  and  odd  vertex pairings in finite graphs.
{\em Canad. J. Math.}, 12:555--567, (1960).

\end{thebibliography}
\end{document}